\documentclass[12pt,a4paper,oneside]{article} 
\usepackage[active]{srcltx} 
\usepackage{amsmath}
\usepackage{amssymb}
\usepackage{wasysym} 
\usepackage[all]{xy}
\usepackage{amsthm}

\newtheorem{theorem}{Theorem}[section]
\newtheorem{lemma}[theorem]{Lemma} 
\newtheorem{remark}[theorem]{Remark}
\newtheorem{corollary}[theorem]{Corollary}  
\newtheorem{definition}[theorem]{Definition} 
\numberwithin{equation}{section}

\newenvironment{zalu}{{\bf Assumption} $\boldsymbol{[\varphi]}$\hspace{0.2cm}}

\newenvironment{zalc}{{\bf Assumption} $\boldsymbol{[c]}$\hspace{0.2cm}}

\newenvironment{zall}{{\bf Assumption} $\boldsymbol{[\lambda]}$\hspace{0.2cm}}

\newenvironment{zalW}{{\bf Assumption} $\boldsymbol{[W]}$\hspace{0.2cm}}

\newenvironment{zalK}{{\bf Assumption} $\boldsymbol{[K]}$\hspace{0.2cm}}

\newenvironment{zalfun}{{\bf Assumption} $\boldsymbol{[L,M]}$\hspace{0.2cm}}

\newcommand{\rr}{\mathbb{R}}

\begin{document}

\title{Existence and uniqueness of solutions for single\\-population McKendrick-von Foerster
models\\with renewal}

\author{Agnieszka Bart{\l}omiejczyk\\
Faculty of Applied Physics and Mathematics\\
Gda{\'{n}}sk University of Tech\-no\-lo\-gy\\
Gabriela Narutowicza 11/12, 80-233 Gda{\'{n}}sk, Poland\\
e-mail: agnes@mif.pg.gda.pl
\and 
Henryk Leszczy\'nski\\
Institute of Mathematics, University of Gda{\'{n}}sk\\
Wita Stwosza 57, 80-952 Gda{\'{n}}sk, Poland\\
e-mail: hleszcz@mat.ug.edu.pl
\and
Piotr Zwierkowski\\
Institute of Mathematics, University of Gda{\'{n}}sk\\
Wita Stwosza 57, 80-952 Gda{\'{n}}sk, Poland\\
e-mail: zwierkow@mat.ug.edu.pl}

\date{}

\maketitle

\begin{abstract}
We study a McKendrick-von Foerster type equation with renewal. 
This model is represented by a single equation which describes one species which produces young individuals. 
The renewal condition is linear, but takes into account some history of the population.
This model addresses nonlocal interactions between individuals structured by age. 
The vast majority of size-structured models are also treatable. 
Our model generalizes a number of earlier models with delays and integrals. 
The existence and uniqueness is proved through a fixed-point approach to an equivalent integral problem in $L^{\infty}\cap L^1.$
\end{abstract}  



%%%%%%%%%%%%%%%%%%%%%%%%%%%%%%%%%%%%%%%%%%%%%%%% Introduction %%%%%%%%%%%%%%%%%%%%%%%%%%%%%%%%%
\section{Introduction}
Von Foerster-McKendrick models (originated in \cite{VF}) describe populations with structure given by age \cite{GMC}, size \cite{Kato} or level of maturation of individuals \cite{LMW}.  
In the literature there are discrete models of that type with finite \cite{lesl} or infinite matrices \cite{PSW}.

We consider one population with a structure given by the size of individual members or level of maturity and with the birth process expressed by a linear renewal equation. An elementary outline of such models, together with their biological interpretation, is provided in \cite{BC}. The best general reference here is the seminal work \cite{MD} and, for the case of age structure, the books \cite{EP, Webb}. There are a number of existence and uniqueness proofs in literature for different versions of the McKendrick-von Foerster equations and which extend size and age-structured problems, e.g. \cite{KWB, Cus1}. Closer to the techniques used in the paper are the papers \cite{CS}, \cite{Kraev} and \cite{Kato}.
In \cite{opt-cont} a~model concerning demographic and economics problems of ageing populations is studied. The model consists of two McKendrick type equations: for a~population and for a~capital stock. 

The governing equation of a structured population is formulated either in the conservation law form $$\frac{\partial u}{\partial t}+\frac{\partial (cu)}{\partial x}=\tilde\lambda u $$
or in the standard form of a hyperbolic equation 
$$\quad\frac{\partial u}{\partial t}+c\frac{\partial u}{\partial x}=\lambda u.$$
These equations are closely related to each other, it suffices to put $\tilde\lambda=\lambda+\frac{\partial c}{\partial x}.$ In our work $\tilde\lambda$ is denoted by $W$ and it is associated with the change of variables in the integral $\int u\,dx.$ The occurrence of $\frac{\partial c}{\partial x}$ follows from the Liouville theorem. This change of variables shows the dynamics of the initial mass transport $\int\varphi\,\exp(\int_0^t W)\,dx.$ 
The global Lipschitz condition for $c$ and $\lambda$ is not sufficient for the global existence and uniqueness, because the nonlinearity $u\,\lambda$ is strong, e.g. $\frac{\partial u}{\partial t}+\frac{\partial u}{\partial x}= u^2$ possesses local solutions. Global existence is due to boundedness of $c$ and $\lambda.$
This assumption is reasonable and commonly used for these terms.

We continue the sequence of results \cite{DLo} and \cite{lzw, lesz}, which are focused on integral fixed-point equations, generated by the differential-functional problems. As a main tool we construct integral fixed-point 
equations and a functional space, invariant with respect to these equations. 
The space consists of $u$ and $z$ describing densities and sizes, respectively. These densities are absolutely continuous in $t$ and Lipschitz continuous in $x,$ and the total sizes are continuous. The Banach contraction principle is applied in this functional space. The renewal condition causes serious problems for any fixed point theorem, not to mention a functional dependence. For instance, \cite{Per} deals with a simple McKendrick-von Foerster model without functionals,
but the Banach fixed point theorem demands some sophisticated technicalities. 

We formulate the differential problem. Let $a>0$ and denote $E=[0,a]\times\rr_+$ and $E_a=[-\tau,a]\times\rr_+,$
where $\rr_+=[0,+\infty).$ If $t\in[0,a]$ and $z\colon[-\tau,a]\to\rr_+,$ then the Hale functional $z_t$ is given by $z_t(s)=z(t+s)$ for $s\in[-\tau,0],$ see \cite{Hale}. If $u\colon E_a\to\rr_+,$ then we consider a natural family of Hale functionals $u_t(\cdot,x)\colon[-\tau,0]\to\rr_+$ for $x\in \rr_+,$ defined by $u_t(s,x)=u(t+s,x)$ for $s\in[-\tau,0]$ (this is the same Hale functional with the parameter $x$).
Suppose that $c\colon E\times{\mathcal C}_+\to\rr$ and $\lambda\colon E\times{\mathcal C}_+\times{\mathcal C}_+\to\rr,$ where ${\mathcal C}_+$ is the positive cone of the space of continuous functions from $[-\tau,0]$ into $\rr_+.$ Let $\varphi\colon E_0\to\rr,$ where $E_0=[-\tau,0]\times\rr_+.$
Consider the differential-functional equation
\begin{equation}
\label{rr}
\partial_t u(t,x)+c\left(t,x,z_t\right)\,\partial_x u(t,x)=u(t,x)\,\lambda\left(t,x,u_{t}(\cdot,x),z_t\right)
\end{equation}
with the initial condition
\begin{equation}\label{wp} 
u(t,x)=\varphi(t,x)\quad \mbox{for}\quad (t,x)\in{}E_0,
\end{equation}
and the renewal condition
\begin{equation}
\label{odn}
u(t,0)=\int_0^\infty K(t,x)\,u_t(\cdot,x)\,dx\quad \mbox{for}\quad t\in{}[0,a],
\end{equation}
where $K\colon E\to{\mathcal C}_{+}^{\ast},$ ${\mathcal C}_{+}^\ast$ is the cone of positive continuous functionals over ${\mathcal C}_+$ and
\begin{equation}
\label{zu} 
z(t)= \int_0^{\infty} u(t,x)\,dx\quad\mbox{for}\quad t\in[-\tau,a].
\end{equation}
Since $u=\varphi$ on $E_0,$ the well posedness of the problem requires the following consistency condition 
\begin{equation*}
\varphi(0,0)=\int_0^\infty K(0,x)\,\varphi(\cdot,x)\,dx,
\end{equation*}
which is valid throughout the paper. We illustrate the functional dependence appearing in the right sides of equations (\ref{rr}) and (\ref{odn}) by several examples:
\begin{enumerate}
\item classical age structured models without delays or integrals:
$$-u(t,x)\,\mu(t,x,u(t,x),z(t))\quad \mbox{and} \quad\int_0^\infty \tilde K(t,x)\,u(t,x)\,dx,$$ where $\mu\colon E\times\rr_+\times\rr_+\to\rr$ and $\tilde K\colon E\to\rr_+,$ \cite{GMC1, GMC},
\item delayed structures, where death and birth rates depend on certain past states of $u$ and $z:$ 
$$-u(t,x)\,\mu(t,x,u(t-\tau,x),z(t-\tau)) \quad\mbox{and} \quad\int_0^\infty \tilde K(t,x)\,u(t-\tau,x)\,dx$$ with the same functions $\tilde K$, $\mu,$ \cite{MRud, Tch},
\item moving averages for densities and total sizes, typical for cumulation effects in mathematical biology and medicine:
\begin{align*}
-u(t,x)\,\mu\left(t,x,\frac{1}{\tau}\int_{t-\tau}^{t}u(s,x)\,ds,\frac{1}{\tau}\int_{t-\tau}^{t}z(s)\,ds\right)
\end{align*}
and \begin{align*}\int_0^\infty\tilde K(t,x)\frac{1}{\tau}\int_{t-\tau}^{t}u(s,x)ds dx\end{align*}
with the same functions $\tilde K$, $\mu,$ \cite{DPos},
\item the size structured model:
$$-u(t,x)\left[\mu(x,z(t))+\partial_x\gamma(x,z(t))\right],$$
where $\mu$ and $\gamma$ denote the mortality and growth rates of individuals, respectively, and $K\equiv0,$ \cite{Far}.
\end{enumerate}
In the literature one can observe various differential-functional models: classical arguments \cite{GMC1, GMC}, delays \cite{Cus, GK, Kuang}, integrals \cite{DPos}, Hale functionals $z_t$ \cite{Hale, KM}, mixed-types (e.g. $z_t$ and its multidimensional generalization \cite{Wu}). We have chosen a unified approach to both functional arguments by means of one-dimensional Hale functional, applied to $u$ and $z.$ Then the description becomes simple but sufficiently general. This approach, despite its simplicity, is surprisingly feasible in mathematical biology. It is worth mentioning that the results of our paper are new even for classical arguments. Our paper is the first work in which there are such unified Hale functionals for $u$ and $z.$
One can raise the question whether anything is missing when solutions $u,$ $z$ are considered in the subclass of continuous functions instead the whole space $L^{\infty}\cap L^1,$ while it is more natural to consider integrable densities $u$ in the biological modeling. The main reason of this restriction lies in the method of the proof where the class of continuous functions seems to be unavoidable. On the other hand, 
having proved existence results for continuous functions, we can consider problem (\ref{rr})--(\ref{zu}) with the initial function of the class $L^1.$ This initial function is approximated by continuous functions of the class $L^1.$ This gives a sequence of approximate solutions which converges weakly to the unique solution of our problem. Since our paper is very extensive and technical, we do not intend to provide any details of this corollary.

The aim of this paper is to look for Carath\'eodory's solutions to (\ref{rr})--(\ref{zu}), i.e.~continuous functions $u\colon E_a\to\rr_+$ which satisfy (\ref{rr}) almost everywhere on $E,$ their derivatives $\partial_tu,$ $\partial_xu$ exist almost everywhere on $E,$ and conditions (\ref{wp}), (\ref{zu}) hold. Condition (\ref{odn}) can be regarded as a definition of $z(t)$ by means of $u(t,\cdot).$ In the present paper we understand that the solutions to problem (\ref{rr})--(\ref{zu}) consist of pairs $(u,z)$ such that $u$ is the Carath\'eodory solution and $z$ is given by (\ref{zu}). We recall its biological interpretation: $(u,z)$ means the density and total size of the population. We focus on equivalent integral equations. 

The paper is organized as follows. In Section \ref{sect-bichar} we introduce bicha\-rac\-te\-ris\-tics of the hyperbolic equation and give their basic properties. We also formulate main assumptions and define the space of admissible functions, where the solution of problem (\ref{rr})--(\ref{zu}) will be found. In Section \ref{sect-twier} we prove the main existence and uniqueness theorem by virtue of the Banach contraction principle. The space of admissible functions is mapped into the same space. The integral operator is a contraction with respect to a Bielecki type norm. Because of the number of technical details the proofs of auxiliary lemmas are collected in the last section.

%%%%%%%%%%%%%%%%%%%%%%%%%%%%%%%%%%%%%%%%%%%%%%%%%%%%% SECTION 1 %%%%%%%%%%%%%%%%%%%%%%%%%%%%%%%%%%%%%%%%%%%
\section{Preliminaries}\label{sect-bichar}
We start with the formulation of characteristic equations and an analysis of $u$ along these characteristics. For a given continuous function $z\colon[-\tau,a]\to\rr_+,$ consider the characteristic equations for problem (\ref{rr})--(\ref{wp}):
\begin{equation}\label{rch} 
\frac{d}{ds}\eta(s)=c\left(s,\eta(s),z_s\right),\quad\eta(t)=x,
\end{equation}
where $(t,x)\in E.$
Let $\eta(\cdot)=\eta[z](\cdot;t,x)$ be the characteristic curve passing through the point $(t,x)\in E,$ i.e. the solution of (\ref{rch}) in the Carath\'eodory sense, cf \cite{Walter}.
We denote the maximal existence interval of $\eta[z](\cdot;t,x)$ by $[\alpha,a],$ where $\alpha=\alpha[z](t,x).$ 
It is clear that either $\alpha=0$ or $\alpha\in(0,a].$ 
If $\alpha=0,$ then the  characteristic curve starts from $(0,\eta(0)).$ 
If $\alpha>0,$ then it starts from $(\alpha,0).$
Equation (\ref{rr}) along a characteristic $\eta(\cdot)=\eta[z](\cdot;t,x)$ is rewritten in the form
\begin{equation}\label{rr-bich}
\frac{d}{ds}u(s,\eta(s))=u(s,\eta(s))\,\lambda(s,\eta(s),u_s(\cdot,\eta(s)),z_s)
\end{equation}
with the initial and boundary condition
\begin{equation*}\label{wp-b}
u(\alpha,\eta(\alpha))=\left\{
\begin{array}{cll}
\varphi(\alpha,\eta(\alpha)) &\mbox{for}& \alpha=0,\\[4pt]
\displaystyle\int_0^{\infty}K(\alpha,y)\,u_t(\alpha,y)\,dy &\mbox{for}& \alpha>0.
\end{array}\right.
\end{equation*}
If $\alpha=\alpha[z](t,x)=0$, then equation (\ref{rr-bich}) is  accompanied by the initial condition (\ref{wp}).
If $\alpha=\alpha[z](t,x)>0$, then (\ref{rr-bich}) is equipped with the boundary condition (\ref{odn}).
Denote by $\eta_0$ the characteristic which starts from $(0,0),$ i.e. $\eta_0(t)=\eta[z](t;0,0).$
\begin{remark}
The differential equation (\ref{rch}) leads to the integral equation 
\begin{equation}\label{rr-calch}
\eta[z](s;t,x)=x-\int_s^t c(\zeta,\eta[z](\zeta;t,x),z_\zeta)\,d\zeta.
\end{equation} 
\end{remark}

We denote by $C_b(X,Y)$ the space of all continuous and bounded functions. By $L^1(X,Y)$ we understand the space of all integrable functions with a natural $L^1$-norm, denoted by $\|\cdot\|_1.$ The symbol $\|\cdot\|$ stands for any supremum norm. We need the following assumptions.
\medskip

\begin{zalu} Suppose that:
\begin{enumerate}
\item 
$\varphi\in{}C_b\left(E_0,\rr_+\right),$ $\varphi(t,\cdot)\in L^1(\rr_+,\rr_+)$ for $t\in[-\tau,0],$ and the function 
$[-\tau,0]\ni t\longmapsto\int_0^\infty\varphi(t,x)\,dx$ is continuous and
$$\|\varphi\|_{\infty,1}:=\int_{0}^{\infty}\sup_{t'\in[-\tau,0]}\varphi(t',x)\,dx<\infty.$$
\item 
there is a constant $L_{\varphi}>0$ such that:
\[|\varphi(t,\bar x)-\varphi(t,x)|\leq L_{\varphi}|\bar x-x|\quad\mbox{on}\quad E_0.\]
\end{enumerate}
\end{zalu}
\medskip

\begin{zalc}
Suppose that:
\begin{enumerate}
\item 
the function $c\colon E\times{\mathcal C}_+\to\rr$ is bounded and measurable in $t\in[0,a]$ for every $(x,q)\in\rr_+\times{\mathcal C}_+,$  
\item 
$|c(t,\bar{x},\bar{q})-c(t,x,q)|\leq L_c(t)\left(|\bar{x}-x|+\|\bar{q}-q\|\right)$ on $E\times{\mathcal C}_+,$ where $L_c\in L^1([0,a],\rr_+),$
\item 
the function $c$ satisfies the estimates
\[\hat c(t)\geq \left\|c(t,{x},{q})\right\| \quad\mbox{on}\quad E\times{\mathcal C}_+\]
and
\[c(t,x,q)\geq \varepsilon_0\hat c(t)\quad\mbox{for}\quad 0\leq t\leq a,\quad 0\leq x\leq \int_0^t\hat{c}(s)\,ds\] 
with some $\varepsilon_0\in(0,1)$ and $\hat c\in L^1([0,a],\rr_+).$
\end{enumerate}
\end{zalc}
\medskip

Concerning Assumption  $[c,3],$ notice that the function $c$ must be strictly positive near the lateral boundary, in reach of characteristic curves which start from this boundary. The condition $c\geq 0$ is natural in mathematical biology, when $c$ describes ageing or maturation, see examples in \cite{GMC}, \cite{LMW}. In size-structured population models the renewal condition may occur not only at 0, but also at other points. All newborns have the same age $0,$ but not the same size.
Our general model refers to the classical Kermack-McKendrik-von Foerster model with $c=\mbox{const}>0,$ \cite{GMC1}. Assumption $[c,3]$ is not satisfied in the Lasota model (\cite{LMW}), where $c(t,x,q)=x,$ hence $c(t,0,q)=0.$ Our results can be generalized to the case of the nonlocal and nonlinear renewal condition 
$$c(t,0,z_t)u(t,0)=\int K(t,x,z_t)u_t(\cdot,x)dx,$$
see \cite{CS}.
Under our assumption $c(t,0,q)>0$ the coefficient $c(t,0,z_t)$ can be incorporated in $K,$ hence we can write it as follows
$$u(t,0)=\int K(t,x,z_t)u_t(\cdot,x)\,dx.$$
The results of our paper carry over to equation (\ref{rr}) with such renewal conditions. Due to the large number of details we omit the generalization.

\medskip

\begin{zall}
Suppose that the function $\lambda\colon E\times{\mathcal C}_+\times{\mathcal C}_+\to\rr$ satisfies the conditions:
\begin{enumerate}
\item 
$\lambda$ is Lebesgue integrable in $t\in[0,a]$ for every $(x,w,q)\in\rr_+\times{\mathcal C}_+\times{\mathcal C}_+,$ 
\item 
there is $L_\lambda\in L^1([0,a],\rr_+)$ such that 
$$|\lambda(t,\bar{x},\bar{w},\bar{q})-\lambda(t,x,w,q)|\leq L_\lambda(t)(|\bar{x}-x|+\|\bar{w}-w\|+\|\bar{q}-q\|)$$
for $t\in[0,a],$ $x,\bar x\in\rr_+,$ $w,\bar w,q,\bar q\in{\mathcal C}_+,$
\item 
there exists $M_\lambda\in{}L^1([0,a],\rr_+)$ such that
$$|\lambda(t,x,w,q)|\leq{}M_\lambda(t)\quad\mbox{for}\quad(t,x){}\in{}E,\,w,q{}\in{\mathcal C}_+.$$
\end{enumerate}
\end{zall}
\medskip

\noindent Denote
\begin{equation}\label{funW}
W(t,x,w,q)=\lambda(t,x,w,q)+\partial_x c(t,x,q)
\end{equation}
for $(t,x)\in E,$ $w,q\in{\mathcal C}_+.$

\medskip

\begin{zalW} Suppose that $W\colon E\times{\mathcal C}_+\times{\mathcal C}_+\to \rr$ satisfies the conditions:
\begin{enumerate}
\item 
$W$ is Lebesgue integrable in $t\in[0,a]$ for every $(x,w,q)\in\rr_+\times{\mathcal C}_+\times{\mathcal C}_+,$ 
\item 
there exists $L_W\in{}L^1\left([0,a],\rr_+\right)$ such that
\[ |W(t,\bar{x},\bar{w},\bar{q})-W(t,x,w,q)|\leq L_W(t)(|\bar{x}-x|+\|\bar{w}-w\|+\|\bar{q}-q\|)\]
for $t\in[0,a],$ $x,\bar x\in\rr_+,$ $w,\bar w,q,\bar q\in{\mathcal C}_+,$
\item 
there is $M_W\in{}L^1([0,a],\rr_+)$ such that
$$ |W(t,x,w,q)|\leq{}M_W(t)\quad \mbox{for}\quad(t,x)\in E,\, w,q\in{\mathcal C}_+.$$
\end{enumerate}
\end{zalW}
\medskip

\begin{zalK} Suppose that: 
\begin{enumerate}
\item 
$\|K(t,x)\|_{{\mathcal C}_{+}^\ast}\leq\kappa$ on $E$ for some constant $\kappa\in\rr_+,$ where $\|\cdot\|_{{\mathcal C}_{+}^\ast}$ is the standard functional norm, 
\item 
$K$ is absolutely continuous on $E$ in the following sense 
\[ \|K(\bar{t},\bar{x})-K(t,x)\|_{{\mathcal C}_{+}^{\ast}} \leq L_K\left[\left|\int_{t}^{\bar t}\hat c(s)\,ds\right|+|\bar x-x|\right] \]
with some $L_K\in\rr_+$ and the function $\hat c$ from Assumption $[c].$
\end{enumerate}
\end{zalK}

\begin{zalfun} The functions $L_c/\hat c,$ $L_\lambda/\hat c,$ $L_W/\hat c,$ $M_\lambda/\hat c,$ $M_W/\hat c$ are bounded on $[0,a].$
\end{zalfun}

\begin{remark}
The most convenient realization of Assumption $[L,M]$ can be achieved by the substitution $L_c=\mbox{const.}\,\hat c,$ $L_\lambda=\mbox{const.}\,\hat c,$ $L_W=\mbox{const.}\,\hat c,$ $M_\lambda=\mbox{const.}\,\hat c,$  $M_W=\mbox{const.}\,\hat c.$ In particular, one can put $L_c=L_\lambda=L_W=M_\lambda=M_W=\hat c.$
\end{remark}

%%%%%%%%%%%%%%%%%%%%%%%%%%%%%%%%%%%%%% SECTION 2 %%%%%%%%%%%%%%%%%%%%%%%%%%%%%%%%%%%%%%%%%%%%%%%%%%%%%%%%%%%%%%
\section{Main results}\label{sect-twier}

Now we are ready to define a space ${\mathcal X}$ of admissible functions in terms of constant from the previous assumptions, where a priori estimates of solutions to (\ref{rr})--(\ref{zu}) are fulfilled.
\begin{definition}\label{defX}
We say that a pair $(u,z)$ belongs to ${\mathcal X}$ iff $u\colon E_a\to\rr_+,$ $z\colon [-\tau,a]\to\rr_+$ are continuous and
\begin{enumerate}
\item $(u,z)$ satisfies conditions (\ref{wp}) and (\ref{zu}),
\item $u(t,0)\leq \kappa\, Z(t),$ $z(t)\leq Z(t),$ $u(t,x)\leq U(t)$ for $t\in[0,a]$ and $x\in\rr_+,$ where 
\begin{eqnarray*}
Z(t)&=&\|\varphi\|_{\infty,1}\,\exp\left(\int_0^t\left[\kappa\,\hat{c}(s)+M_W(s)\right]\,ds\right),\\
U(t)&=& \exp\left(\int_0^t M_\lambda(s)\,ds\right)\,\max\left\{\|\varphi\|,\kappa \,Z(t)\right\},
\end{eqnarray*}
\item $|u(\bar{t},0)-u(t,0)|\leq G_u\,\int_t^{\bar{t}}\hat{c}(s)\,ds$ for $0\leq t\leq\bar t\leq a,$ where 
\begin{multline*}
G_u=\Biggl\{{\kappa}^2 Z(a)+\left[2L_K+{\kappa}^2 Z(a)\left\|\frac{M_W}{\hat c}\right\|\right]\\
\times\left[\int_0^a\hat c(s)\,ds +\|\varphi\|_{\infty,1}\right]\Biggr\}\,\exp\left(\int_0^a M_W(s)\,ds\right),
\end{multline*}
\item $|u(t,\bar x)-u(t,x)| \leq L_u(t)\,|\bar x-x|$ for $t\in{}[0,a]$ and $x,\bar x\in{}\rr_+,$ where $L_u$ is given by
\begin{align*}
&L_u(t)=-1+(1+L_u(0))\\
&\times\exp\left\{\left[{\kappa} Z(a)+\|\varphi\|\right]\int_0^tL_\lambda(s)\,ds\,\exp\left(\int_0^a(L_c(s)+M_\lambda(s))\,ds\right)\right\},
\end{align*}
where $L_u(0)=\left\{L_{\varphi}+\frac{G_u+\kappa Z(a)\|M_\lambda/\hat{c}\|}{\varepsilon_0}\right\}\,\exp\left(\int_0^a\left(L_c(s)+M_\lambda(s)\right)\,ds\right).$
\end{enumerate}
\end{definition}
\begin{remark}
If $u\colon E_a\to\rr_+$ is a bounded, continuous and integrable function, then the corresponding function $z\colon[-\tau,a]\to\rr_+$ is measurable and bounded. However, for $(u,z)\in{\mathcal X}$ the function $z$ has an enhanced regularity, so that it becomes absolutely continuous on $[0,a].$ In fact, the function $z$ is as regular as $u(\cdot,0)$ on $[0,a].$
\end{remark}

Let us formulate the main result of this paper.

\begin{theorem}\label{pktstaly}
Suppose that Assumptions $[\varphi],$ $[c],$ $[\lambda],$ $[W],$ $[K]$ and $[L,M]$ are satisfied. Then there exists exactly one solution $(u,z)$ of problem (\ref{rr})--(\ref{zu}) in the class $\mathcal X.$
\end{theorem}

Our main existence theorem will be proved by means of the Banach contraction principle in the space $\mathcal X$. Suppose that $(u,z)\in {\mathcal X}.$ 
We construct a new pair of functions $(\tilde{u},\tilde{z})$ via the renewal condition (\ref{odn}) as follows. Suppose that $\eta=\eta[z](\cdot;t,x)$ is a characteristic determined by the Cauchy problem (\ref{rch}). Denote
\[ P_{t,x}^{u,z}(s) :=
\left(s,\eta[z](s;t,x),u_s(\cdot,\eta[z](s;t,x)),z_s\right). \]
Let $\tilde u(t,0)$ for $t\in[0,a]$ be the solution of the following Volterra integral equation
\begin{equation}\label{rr-cal0}
\begin{split}
\tilde u(t,0)&= \int_0^t K(t,\eta[z](t;\xi,0))\,\tilde u_\xi(\cdot,0)\,c(\xi,0,z_\xi)\,\exp\left(\int_\xi^t W\left(P_{\xi,0}^{u,z}(s)\right)\,ds\right)\,d\xi\\
&\quad +\,\int^{\infty}_0 K(t,\eta[z](t;0,y))\,\varphi(\cdot,y)\,\exp{\left(\int_0^{t}W\left(P_{0,y}^{u,z}(s)\right)\,ds\right)}\,dy,
\end{split}
\end{equation}
where $W$ is defined by (\ref{funW}). The explanation of the changes of variables $x\mapsto \xi$ and $x\mapsto y$ is given in Remark \ref{rem_warbrzeg} at the very end of the paper.
Because the Volterra integral equation (\ref{rr-cal0}) provides a natural boundary condition, the function $\tilde u$ on the whole set $E$ will be the only solution of the PDE 
\begin{equation*}
\partial_t \tilde u(t,x)+c\left(t,x,z_t\right)\,\partial_x\tilde u(t,x)=\tilde u(t,x)\,\lambda\left(t,x,u_t(\cdot,x),z_t\right)
\end{equation*}
with the initial condition $\tilde u=\varphi$ on $E_0.$ Considering this problem along the characteristics satisfying (\ref{rch}), we get its solution by the explicit formula
\begin{equation}\label{rr-cal}
\tilde u(t,x)=\tilde u(\alpha,\eta[z](\alpha;t,x))\,\exp\left(\int_{\alpha}^t\lambda\left(P_{t,x}^{u,z}(s)\right)\,ds\right)
\end{equation}
for $(t,x)\in E,$ where $\alpha=\alpha[z](t,x).$
According to (\ref{zu}), we have 
$$\tilde z(t) =\int_{0}^{\infty} \tilde u(t,x)\,dx\quad\mbox{for}\quad t\in [-\tau,a].$$ 
This way we have constructed an integral operator $\mathcal T $ which maps a pair of functions $(u,z)$ to a pair $(\tilde u,\tilde z)=\mathcal{T}(u,z).$
By virtue of the Banach contraction principle we show that the operator $\mathcal T$ has exactly one fixed point $(u,z)\in \mathcal X.$ This fixed point satisfies the differential-functional problem (\ref{rr})--(\ref{zu}). This goal will be achieved in three auxiliary lemmas:
\begin{enumerate}
\item if $(u,z)\in \mathcal{X},$ then $\mathcal{T}(u,z)$ satisfies the conditions $1$--$2$ of Definition \ref{defX},
\item $\mathcal{T}(u,z)$ satisfies the conditions $3$--$4$ of Definition \ref{defX},
\item the operator $\mathcal{T}\colon\mathcal{X}\to\mathcal{X}$ is a contraction. 
\end{enumerate}
Because of multitudes of technical details we relegate the proofs of these lemmas to the next section.
\begin{lemma}\label{war-1-3} Suppose that Assumptions $[c,3],$ $[\lambda,3],$ $[W,3]$ and $[K,1]$ are satisfied.
If $(u,z)\in\mathcal X,$ then $\tilde u$ is bounded, continuous, the pair $(\tilde u,\tilde z)=\mathcal{T}(u,z)$ satisfies condition (\ref{zu}), and the following estimates hold true:
\[0\leq \tilde u(t,0)\leq \kappa Z(t),\quad \tilde z(t)\leq Z(t),\quad \tilde u(t,x)\leq U(t)\quad \mbox{for}\quad(t,x)\in E.\]
\end{lemma}
\begin{lemma}\label{war-4-5} Suppose that Assumptions $[\varphi],$ $[c],$ $[\lambda],$ $[W],$ $[K]$ and $[L,M]$ are satisfied. If $(u,z)\in\mathcal X,$ then the pair $(\tilde u,\tilde z)=\mathcal{T}(u,z)$ satisfies the conditions 
\begin{align*}
|\tilde u(\bar t,0)-\tilde u(t,0)|&\leq G_u\int_t^{\bar t}\hat c (s)\,ds\quad\mbox{for} \quad 0\leq t\leq\bar t\leq a,\\
|\tilde u(t,\bar x)-\tilde u(t,x)&|\leq L_u(t)\,|\bar x-x| \quad\mbox{for}\quad t\in[0,a],\,\bar x,x\in\rr_+,
\end{align*}
where $G_u$ and $L_u(t)$ are the same as in Definition \ref{defX}, $3$--$4$.
 
\end{lemma}

\begin{definition} The Bielecki norm is given by 
$$\|(u,z)\|_B=\max\{\|u/B\|,\|z/B\|\},$$ where $B=B(t),$ $B\colon\rr_+\to\rr_+$ is a positive, continuous, nondecreasing function. The meaning of the supremum norms $\|u/B\|$ and $\|z/B\|$ is obvious, see \cite{B}.
\end{definition}
\begin{lemma}\label{zwez}
Suppose that Assumptions $[\varphi],$ $[c],$ $[\lambda],$ $[W],$ $[K]$ and $[L,M]$ are satisfied. Then the operator ${\mathcal T}\colon{\mathcal X}\to\mathcal{X}$ is a contraction with respect to a~Bie\-lecki norm $\|\cdot\|_B$ for some $B\colon\rr_+\to\rr_+,$ that is: there is $\Theta\in(0,1)$ such that 
$$\|\mathcal{T}(\bar u,\bar z)-\mathcal{T}(u,z)\|_B\leq \Theta \,\|(\bar u,\bar z)-(u,z)\|_B\quad\mbox{on}\quad\mathcal{X}.$$
In fact, for any $\Theta\in(0,1)$ we can find a function $B$ of the form $B(t)=\exp\left(C\int_0^t\hat c(s)\,ds\right)$ such that the above contraction inequality holds true.
\end{lemma}

\begin{corollary}\label{stalelip}
If the functions $\lambda,$ $c,$ $\partial_x c$ are bounded and continuous; $\lambda$ is Lipschitz continuous in $x,p,q$;  $c$ and $\partial_x c$ are Lipschitz continuous in $x,q;$ $K$ is nonnegative and Lipschitz continuous; $c$ is nonnegative on $E$; $c(t,x,q)\geq \varepsilon_1>0$ for all $t\in [0,a]$ and $x\in[0,tx_0/a]$ with some $x_0>0$; $\varphi$ is Lipschitz continuous and satisfies Assumption $[\varphi]$, then there is a solution to (\ref{rr})--(\ref{zu}) which is Lipschitz continuous with respect to both variables $t,$ $x.$
\end{corollary}

\begin{proof}
All assumptions of Theorem \ref{pktstaly} are satisfied with the functions $M_\lambda,$ $M_W,$ $L_\lambda,$ $L_W,$ $\hat c,$ $L_c$ that are constant. Observe that $\varepsilon_0=\varepsilon_1/{\hat c}$ that is Assumption $[c,3]$ is also satisfied. Since the functions mentioned above are constant, the functions $\alpha[z],$ $\eta[z]$ and $u(\cdot,0)$ inherit the Lipschitz continuity with respect to $t.$
Therefore, using (\ref{rr-cal}), we obtain that $u$ is Lipschitz continuous.
\end{proof}

\begin{corollary}
Suppose that the assumptions of Corollary \ref{stalelip} are satisfied. If $\varphi,$ $c$ and $K$ are $C^1$ functions, then the solution of (\ref{rr})--(\ref{zu}) is a $C^1$ function in the whole domain $E,$ except the characteristic curve which starts from $(0,0).$
\end{corollary}

\begin{proof}
Since the initial function $\varphi$ is of the class $C^1$ and $(\tilde u,\tilde z)={\mathcal T}(u,z)\in{\mathcal X},$ it follows from (\ref{rr-cal}) that $\tilde u$ is of the class $C^1$ for $(t,x)$ such that $x>\eta_0(t),$ $t\in[0,a].$ Due to the Volterra integral equation (\ref{rr-cal0}), $C^1$--regularity of $K$ results in the same regularity of $\tilde u$ for $x<\eta_0(t),$ $t\in[0,a],$ provided that $\alpha[z]$ is $C^1,$ which follows from the smoothness property of the function $c.$
\end{proof}

%%%%%%%%%%%%%%%%%%%%%%%%%%%%%%%%%%%%%%%%%%%%%%%%%%% SECTION 4 %%%%%%%%%%%%%%%%%%%%%%%%%%%%%%%%%%%%%%%%%%%%%
\section{Proofs of Lemmas}
\begin{proof}[Proof of Lemma \ref{war-1-3}]
Let $t\in[0,a].$ It follows from the Volterra equation (\ref{rr-cal0}) that
\begin{eqnarray*}
\tilde{u}(t,0)&\leq&\kappa\int_0^t \hat c(\xi)\,\left\|\tilde u_\xi(\cdot,0)\right\|\,\exp\left(\int_\xi^t M_W(s)ds\right)\,d\xi\\
&&+\kappa\,\|\varphi\|_{\infty,1}\,\exp\left(\int_0^t M_W(s)\,ds\right).
\end{eqnarray*}
Since the right-hand side is increasing with respect to $t,$ the left-hand side can be replaced by $\|\tilde{u}_t(\cdot,0)\|.$
Applying the Gronwall lemma, we get the inequality
\begin{equation*}
\|\tilde{u}_t(\cdot,0)\|\,\exp\left(-\int_0^t M_W(s)\,ds\right) \leq \kappa\,\|\varphi\|_{\infty,1}\,\exp\left(\int_0^t \kappa\hat c(s)\,ds\right).
\end{equation*}
Therefore, we have $\tilde u(t,0)\leq \kappa Z(t),$ where $Z(t)$ is described in Definition \ref{defX}, $2.$ Since $\tilde z$ has the estimate
\begin{equation*}
\tilde z(t)\leq\int_0^t\kappa\,Z(\xi)\,\hat c(\xi)\,\exp\left(\int_\xi^t M_W(s)\,ds\right)\,d\xi+\|\varphi\|_{\infty,1}\,\exp\left(\int_0^t M_W(s)\,ds\right),
\end{equation*}
it is easy to observe that $\tilde z(t)\leq Z(t).$

Now we show the estimate for $\tilde u(t,x).$ From (\ref{rr-cal}) we deduce that\\
$1)$ if $x\geq\eta_0(t),$ then 
\begin{equation*}
\tilde u(t,x)=\varphi(0,\eta[z](0;t,x))\exp\left(\int_0^t\lambda\left(P_{t,x}^{u,z}(s)\right)ds\right)\leq \|\varphi\|\,\exp\left(\int_0^t M_\lambda(s)\,ds\right),
\end{equation*}
$2)$ if $x\leq\eta_0(t),$ then 
\begin{eqnarray*}
\tilde u(t,x)&=&u(\alpha,0)\,\exp\left(\int_\alpha^t\lambda\left(P_{t,x}^{u,z}(s)\right)\,ds\right)\\
&\leq&\kappa\,Z(\alpha)\,\exp\left(\int_\alpha^t M_\lambda(s)\,ds\right)\\
&=&\kappa\,\|\varphi\|_{\infty,1}\,\exp\left(\int_0^\alpha \left[\kappa\,\hat c(s)+M_W(s)\right]\,ds\right)\,\exp\left(\int_\alpha^t M_\lambda(s)\,ds\right)\\
&\leq& \kappa\,\|\varphi\|_{\infty,1}\,\exp\left(\int_0^t \max\left\{\kappa\,\hat c(s)+M_W(s),\,M_\lambda(s)\right\}\,ds\right),
\end{eqnarray*}
where $\alpha=\alpha[z](t,x).$
Both estimates in cases $1)$ and $2)$ can be unified as follows
\begin{eqnarray*}
\tilde u(t,x)&\leq&\max\left\{\|\varphi\|,\,\kappa\,\|\varphi\|_{\infty,1}\right\}\\
&&\times\exp\left(\int_0^t \max\left\{\kappa\,\hat c(s)+M_W(s),\,M_\lambda(s)\right\}\,ds\right),
\end{eqnarray*}
hence $\tilde u(t,x)\leq U(t)$ on $E.$ Regularity assertions and condition (\ref{zu}) for $\tilde u,$ $\tilde z$ are trivial. This completes the proof.
\end{proof}
\medskip

\begin{proof}[Proof of Lemma \ref{war-4-5}]
Let $(u,z)\in \mathcal{X}.$ In order to demonstrate the Lipschitz condition of $\tilde u(t,\cdot)$ and the absolute continuity of $\tilde u(\cdot,0),$ we analyze properties of $\eta[z]$ and $\alpha[z].$ Take arbitrary $(t,x)\in E$ and $\bar x\in\rr_+.$\\
{\it Step 1.} \emph{Estimate of increments of $\eta$ for $\bar x$ and $x.$} We derive from (\ref{rr-calch}) the integral inequality
\[ |\eta[z](s;t,\bar x)-\eta[z](s;t,x)| \leq |\bar x-x| + \left|\int_s^t L_c(\zeta)\,\left|\eta[z](\zeta;t,\bar x)-\eta[z](\zeta;t,x)\right|\,d\zeta\right|.\]
Applying the Gronwall lemma, we get
\[ |\eta[z](s;t,\bar x)-\eta[z](s;t,x)| \leq |\bar x-x|\,\exp\left(\left|\int_s^t L_c(\zeta)\,d\zeta\right|\right).\]
{\it Step 2.} \emph{Estimate of increments of $\eta$ for $\bar t$ and $t.$} From (\ref{rr-calch}) we get the inequality 
\[
|\eta[z](s;\bar t,x)-\eta[z](s;t,x)|\leq \int_t^{\bar t}\hat c(\zeta)\,d\zeta +\left|\int_s^t \!L_c(\zeta)\left|\eta[z](\zeta;\bar t,x)-\eta[z](\zeta;t,x)\right|\,d\zeta\right|
\] 
for $0\leq t\leq \bar t \leq a.$
Applying the Gronwall lemma, we get 
\[|\eta[z](s;\bar t,x)-\eta[z](s;t,x)| \leq \int_t^{\bar t}\hat c(\zeta)\,d\zeta\,\exp\left(\left|\int_s^tL_c(\zeta)\,d\zeta\right|\right).\]
Similarly, we have
\begin{eqnarray*}
|\eta[z](\bar s;t,x)-\eta[z](s;t,x)| = \left|\int_s^{\bar s} c(\zeta,\eta[z](\zeta;t,x),z_\zeta)\,d\zeta \right|\leq \left|\int_{s}^{\bar s}\hat c(\zeta)\,d\zeta\right|.
\end{eqnarray*}
{\it Step 3.} \emph{Estimate of some integrals.} By the definition of $\alpha=\alpha[z](t,x)$ we have the integral identity
$$0=x-\int_\alpha^tc(\zeta,\eta[z](\zeta;t,x),z_\zeta)\,d\zeta\quad \mbox{for}\quad \alpha>0.$$
Denote $\bar\alpha=\alpha[z](t,\bar x).$ Suppose that $\alpha\leq \bar\alpha.$ Then we have
$$0=\bar x-\int_{\bar\alpha}^{t} c(\zeta,\eta[z](\zeta;t,\bar{x}),z_\zeta)\,d\zeta.$$
If we subtract these identities, then
$$\left|\int_{\alpha}^{\bar\alpha}c(\zeta,\eta[z](\zeta;t,\bar{x}),z_\zeta)\,d\zeta\right|\leq |\bar x-x|+\int_{\alpha}^{t} L_c(\zeta)\left|\eta[z](\zeta;t,\bar{x})-\eta[z](\zeta;t,x)\right|\,d\zeta. 
$$
Applying Assumption $[c,2,3]$ and Step $1$ we get
\begin{eqnarray*}
\varepsilon_0\int_{\alpha}^{\bar\alpha}\hat c(\zeta)\,d\zeta\leq |\bar x-x|+\int_{\alpha}^{t}L_c(\zeta)\,|\bar x-x|\,\exp\left(\int_{\zeta}^{t} L_c(s)ds   \right)\,d\zeta.
\end{eqnarray*}
Consequently, we have
\begin{eqnarray*}
\varepsilon_0\int_{\alpha}^{\bar\alpha}\hat c(\zeta)\,d\zeta\leq |\bar x-x|\,\exp\left(\int_{\alpha}^{t}L_c(\zeta)\,d\zeta\right).
\end{eqnarray*}
{\it Step 4.} \emph{Estimate of increments of $\lambda$ and $W$ along characteristics.} Since $u(s,\cdot)$ is Lipschitz continuous, we get
\begin{eqnarray*}
\left|u(s,\eta[z](s;t,\bar x))-u(s,\eta[z](s;t,x))\right| &\leq& L_u(s)\left|\eta[z](s;t,\bar x)-\eta[z](s;t,x)\right|\\
&\leq& L_u(s)\,|\bar x-x|\,\exp\left(\int_s^tL_c(\zeta)\,d\zeta\right)
\end{eqnarray*}
for $0\leq s\leq t.$
This inequality is applied to the estimates of increments of $\lambda$ and $W$, in particular, we have
\begin{eqnarray*}
\left|\lambda(P_{t,\bar x}^{u,z}(s))-\lambda\left(P_{t,x}^{u,z}(s)\right)\right|&\leq& L_\lambda(s)\,|\eta[z](s;t,\bar{x})-\eta[z](s;t,x)|\\ &&+L_\lambda(s)\,|u(s,\eta[z](s;t,\bar x))-u(s,\eta[z](s;t,x))|\\
&\leq& L_\lambda(s)\,|\bar x- x|\, \left(1+L_u(s)\right)\,\exp\left(\int_s^tL_c(\zeta)\,d\zeta\right)
\end{eqnarray*}
for $0\leq s\leq t.$ The estimate of $\left|W\left(P_{t,\bar x}^{u,z}(s)\right)-W\left(P_{t,x}^{u,z}(s)\right)\right|$ is similar.\\
{\it Step 5.} \emph{Estimate of $\tilde u(\bar t,0)-\tilde u(t,0).$} Take $t\leq\bar t.$ From (\ref{rr-cal0}) we obtain the inequality
\begin{eqnarray*}
\lefteqn{|\tilde u(\bar t,0)-\tilde u(t,0)| \leq
\int_t^{\bar t} \left\|K(\bar t,\eta[z](\bar t;\xi,0))\right\|_{{\mathcal C}_{+}^\ast}\,\|\tilde u_\xi(\cdot,0)\|\,c(\xi,0,z_\xi)}\\
&&\times\exp\left(\int_{\xi}^{\bar t}W\left(P_{\xi,0}^{u,z}(s)\right)\,ds\right)\,d\xi\\
&+&\int_0^t \left\|K(\bar t,\eta[z](\bar t;\xi,0))-K(t,\eta[z](t;\xi,0))\right\|_{{\mathcal C}_{+}^\ast}\,\|\tilde u_\xi(\cdot,0)\|\,c(\xi,0,z_\xi)\\
&&\times\exp\left(\int_{\xi}^{\bar t}W\left(P_{\xi,0}^{u,z}(s)\right)\,ds\right)\,d\xi\\
&+&\int_0^t \left\|K(t,\eta[z](t;\xi,0))\right\|_{{\mathcal C}_{+}^\ast}\,\|\tilde u_\xi(\cdot,0)\|\,c(\xi,0,z_\xi)\\
&&\times\left|\exp\left(\int_{\xi}^{\bar t}W\left(P_{\xi,0}^{u,z}(s)\right)\,ds\right)-\exp\left(\int_{\xi}^{t}W\left(P_{\xi,0}^{u,z}(s)\right)\,ds\right)\right|\,d\xi\\
&+&\int_0^\infty \left\|K(\bar t,\eta[z](\bar t;0,y))-K(t,\eta[z](t;0,y))\right\|_{{\mathcal C}_{+}^\ast}\,\|\varphi(\cdot,y)\|\\
&&\times\exp\left(\int_{0}^{\bar t}W\left(P_{0,y}^{u,z}(s)\right)\,ds\right)\,dy\\
&+&\int_0^\infty \left\|K(t,\eta[z](t;0,y))\right\|_{{\mathcal C}_{+}^\ast}\,\|\varphi(\cdot,y)\|\\
&&\times\left|\exp\left(\int_0^{\bar t}W\left(P_{0,y}^{u,z}(s)\right)\,ds\right)-\exp\left(\int_0^{t}W\left(P_{0,y}^{u,z}(s)\right)\,ds\right)\right|\,dy.
\end{eqnarray*}
Applying Lemma \ref{war-1-3}, Step $4$ and Assumptions $[K],$ $[c,3],$ $[W,3],$ we obtain
\begin{eqnarray*}
\lefteqn{|\tilde u(\bar t,0)-\tilde u(t,0)|\leq\kappa^2\int_t^{\bar t} Z(\xi)\,\hat c(\xi)\, \exp\left(\int_\xi^{\bar t} M_W(s)\,ds\right)\,d\xi}\\
&+& \kappa\,L_K\int_0^t \left[\int_{t}^{\bar t} \hat c(s)\,ds+ |\eta[z](\bar t;\xi,0)-\eta[z](t;\xi,0)|\right]\,Z(\xi)\,\hat c(\xi)\\
&&\times\exp\left(\int_\xi^{\bar t} M_W(s)ds\right)\,d\xi\\
&+&\kappa^2\int_{0}^{t}Z(\xi)\,\hat c(\xi)\,\int_t^{\bar t} \left|W\left(P_{\xi,0}^{u,z}(s)\right)\right|\,ds\,\exp\left(\int_{\xi}^{\bar t}M_W(s)\,ds\right) \,d\xi
\end{eqnarray*}
\begin{eqnarray*}
&+&L_K\int_0^{\infty} \left[\int_t^{\bar t}\hat c(s)\,ds+|\eta[z](\bar t;0,y)-\eta[z](t;0,y)|\right]\,\|\varphi(\cdot,y)\|\\
&&\times\exp\left(\int_{0}^{\bar t}M_W(s)\,ds\right)\,dy\\
&+& \kappa\int_0^{\infty} \|\varphi(\cdot,y)\|\,\int_t^{\bar t} \left|W\left(P_{0,y}^{u,z}(s)\right)\right|\,ds\,\exp\left(\int_{0}^{\bar t}M_W(s)\,ds\right)\,dy.
\end{eqnarray*}
By the last inequality from Step $2$ we arrive at the condition $4$ for $\tilde u:$
\begin{eqnarray*}
|\tilde u(\bar t,0)-\tilde u(t,0)|\leq G_u\int_{t}^{\tilde t}\hat c(s)\,ds,
\end{eqnarray*}
where $G_u$ is the same as in Definition \ref{defX}, $3.$\\
{\it Step 6.} \emph{Estimates of increments of $\tilde u$ for $\bar x$ and $x.$} Denote $\alpha=\alpha[z](t,x)$ and $\bar\alpha=\alpha[z](t,\bar x).$
If $\alpha=\bar\alpha=0,$ then Assumptions $[\varphi],$ $[\lambda,3]$ together with Steps $1$ and $4$, applied to equation (\ref{rr-cal}), imply the estimates
\begin{align*}
&|\tilde u(t,\bar x)-\tilde u(t,x)| \leq L_{\varphi}\,|\eta[z](0;t,\bar x)-\eta[z](0;t,x)| \exp\left(\int_0^tM_\lambda(s)\,ds\right)\\
&+\|\varphi\|\,\exp\left(\int_0^t M_\lambda (s)\,ds\right)\int_0^t L_\lambda(s)\,|\bar x - x|\,(1+L_u(s))\,\exp\left(\int_s^tL_c(\zeta)\,d\zeta\right)\,ds\\
&\leq |\bar x-x|\,\exp\left(\int_0^t \left(L_c(s) + M_\lambda(s)\right)\,ds\right)\left\{L_{\varphi}+\|\varphi\|\,\int_0^t L_\lambda(s)(1+L_u(s))\,ds\right\}.
\end{align*}
If $0<\alpha<\bar\alpha,$ then, utilizing Step $4,$ we obtain
\begin{eqnarray*}
\lefteqn{|\tilde u(t,\bar x)-\tilde u(t,x)|}\\
&\leq& G_u\int_\alpha^{\bar\alpha} \hat c(s)\,ds\,\exp\left(\int_{\bar\alpha}^{t} M_\lambda(s)\,ds\right)
+ \kappa\, Z(\alpha)\,\exp\left(\int_{\alpha}^{t} M_\lambda(s)\,ds\right)\\
&&\times\left\{\int_{\alpha}^{\bar\alpha}M_{\lambda}(s)\,ds+\int_{\bar\alpha}^t L_\lambda(s)\,|\bar x - x|\left(1+L_u(s)\right)\,\exp\left(\int_s^tL_c(\zeta)\,d\zeta\right)\,ds\right\}.
\end{eqnarray*}
By Step $3,$ we have
$$\int_{\alpha}^{\bar\alpha}\hat c(s)\,ds \leq \frac{1}{\varepsilon_0}\,|\bar x-x|\, \exp\left(\int_\alpha^t L_c(s)\,ds\right).$$
Hence we get the inequality
\begin{multline*}
|\tilde u(t,\bar x) - \tilde u(t,x)| \leq |\bar x-x| \exp\left(\int_0^t\left(L_c(s)+M_\lambda(s)\right)\,ds\right)\\
\times\,\left\{\frac{G_u}{\varepsilon_0} +\kappa\, Z(a)\,\left[\frac{\left\|M_{\lambda}/\hat c\right\|}{\varepsilon_0}+\int_0^t L_\lambda(s)\left(1+L_u(s)\right)\,ds\right]\right\}.
\end{multline*}
For arbitrary $x,\bar x\in \rr_+$ one can find an intermediate point $x^\ast$ between them such that the differences $|\tilde u(t,\bar x)-\tilde u(t,x^\ast)|$ and $|\tilde u(t,x^\ast)-\tilde u(t,x)|$ have the upper bounds from the above two cases. Thus we deduce the desired inequality
\begin{eqnarray*}
|\tilde u(t,\bar x) - \tilde u(t,x)| \leq L_u(t)\,|\bar x-x|
\end{eqnarray*}
for all $x,\bar x\in\rr_+,$ where $L_u$ is defined in Definition \ref{defX}, $4.$
\end{proof}
\begin{remark}\label{rem:Lu}
The function $L_u$ from Definition \ref{defX}, $4$ satisfies the following integral equation  
\begin{multline*}
L_u(t)=\exp\left(\int_0^a\left(L_c(s)+M_\lambda(s)\right)\,ds\right)\\
\times \left\{L_{\varphi}+\frac{G_u+\kappa\, Z(a)\,\|M_\lambda/\hat{c}\|}{\varepsilon_0}
+\left[\kappa Z(a)+\|\varphi\|\right] \int_0^t L_\lambda(s)\left(1+L_u(s)\right)\,ds\right\}.
\end{multline*}
\end{remark}
\begin{proof}[Proof of Lemma \ref{zwez}]
Take any $(\bar u,\bar z),$ $(u,z)\in\mathcal{X}$ and $(t,x)\in E.$ Let $B\colon[0,a]\to \rr$ be a positive, continuous and nondecreasing function whose precise specification will be given later.\\
{\it Step 1.} \emph{Estimate of increments of $\eta$ for $\bar z$ and $z.$} By the Gronwall lemma we get 
\begin{equation*}
|\eta[\bar z](s;t,x)-\eta[z](s;t,x)| \leq \|(\bar z-z)/B\| \left|\int_s^tB(\zeta)\,L_c(\zeta)\,d\zeta\,\exp\left(\int_s^t L_c(\zeta)\,d\zeta\right)\right|
\end{equation*} 
whenever $s$ belongs to the domains of both characteristics.\\
{\it Step 2.} \emph{Estimate of some integrals for $\alpha[\bar z]$ and $\alpha[z].$} Denote $\alpha=\alpha[z](t,x)$ and $\bar\alpha=\alpha[\bar z](t,x).$ Assume that $0<\alpha <\bar\alpha.$ Applying (\ref{rr-calch}) to both characteristics and Assumption $[c,3],$ we obtain the estimate
\begin{multline*}
\varepsilon_0\int_{\alpha}^{\bar\alpha}\hat{c}(s)\,ds\leq\int_{\alpha}^{t} L_c(s)\,\left\{\left|\eta[\bar z](s;t,x)-\eta[z](s;t,x)\right|+\|\bar z_s-z_s\|\right\}\,ds \\
\leq \|(\bar z-z)/B\|\,\int_\alpha^t L_c(s)\,\left\{\int_s^t B(\zeta)\,L_c(\zeta)\,d\zeta\,\exp\left(\int_s^t L_c(\zeta)\,d\zeta\right)+B(s)\right\}\,ds.
\end{multline*}
The second inequality is a simple consequence of Step 1. \\
{\it Step 3.} \emph{Estimate of increments of $\lambda$ and $W$ along characteristics.} We start with the difference of $\bar u$ and $u$ taken along their characteristics $\eta[\bar z]$ and $\eta[z].$ Using the function $L_u,$ defined by the formula in Remark \ref{rem:Lu}, we get
\begin{multline*}
|\bar u(s,\eta[\bar z](s;t,x))-u(s,\eta[z](s;t,x))|\leq \|(\bar u-u)/B\|\,B(s)\\
+L_u(s)\,\|(\bar z -z)/B\|\,\left|\int_{s}^{t}B(\zeta)\,L_c(\zeta)\,d\zeta\,\exp\left(\int_{s}^{t}L_c(\zeta)\,d\zeta\right) \right|.
\end{multline*}
By Assumption $[\lambda,2]$ and the above inequality, we obtain
\begin{eqnarray*}
\lefteqn{\left|\lambda\left(P_{t,x}^{\bar u,\bar z}(s)\right)-\lambda\left(P_{t,x}^{u,z}(s)\right)\right|}\\
&\leq& L_\lambda(s)\,\left\{\left(1+L_u(s)\right)\,\|(\bar z-z)/B\| \left|\int_{s}^{t}B(\zeta)L_c(\zeta)\,d\zeta\,\exp\left(\int_{s}^{t}L_c(\zeta)\,d\zeta\right)\right|\right.\\
&&+\left.\vphantom{\int_0^t}\|(\bar u-u)/B\|\,B(s)+\|(\bar z-z)/B\|\,B(s)\right\}.
\end{eqnarray*}
A similar estimate can be derived for increments of $W.$\\
{\it Step 4.} \emph{Estimate of $\tilde{\bar u}(t,0)-\tilde u(t,0).$} 
From (\ref{rr-cal0}) we have 
\begin{align*}
&|\tilde{\bar u}(t,0)-\tilde u(t,0)|\\
&\leq
\kappa\,L_K\int_0^t |\eta[\bar z](t;\xi,0)-\eta[z](t;\xi,0)|\,Z(\xi)\,\hat c(\xi)\,\exp\left(\int_\xi^t M_W(s)\,ds\right)\,d\xi\\
&+\kappa \int_0^t \|{\tilde{\bar u}}_\xi(\cdot,0)-{\tilde u}_\xi(\cdot,0)\|\,\hat c(\xi)\,\exp\left(\int_\xi^tM_W(s)\,ds\right)\,d\xi\\
&+\kappa^2 \int_0^t Z(\xi)\,L_c(\xi)\,\|\bar z_\xi-z_\xi\|\,\exp\left(\int_\xi^tM_W(s)\,ds\right)\,d\xi\\
&+\kappa^2 \int_0^t Z(\xi)\hat c(\xi)\left|\exp\left(\int_\xi^t W\left(P^{\bar u,\bar z}_{\xi,0}(s)\right)\,ds\right)-
\exp\left(\int_\xi^t W\left(P^{u,z}_{\xi,0}(s)\right)\,ds\right)\right|\,d\xi
\end{align*}
\begin{align*}
&+L_K\int_0^\infty |\eta[\bar z](t;0,y)-\eta[z](t;0,y)|\,\|\varphi(\cdot,y)\|\,\exp\left(\int_0^t M_W(s)\,ds\right)\,dy\\
&+\kappa\int_0^\infty \|\varphi(\cdot,y)\|\,\left|\exp\left(\int_0^t W\left(P^{\bar u,\bar z}_{0,y}(s)\right)\,ds\right)-
\exp\left(\int_0^t W\left(P^{u,z}_{0,y}(s)\right)\,ds\right)\right|\,dy.
\end{align*}
Applying the Gronwall inequality, we get
\begin{eqnarray*}
|\tilde{\bar u}(t,0)-\tilde u(t,0)|\leq C_0\,\|(\bar u-u,\bar z-z)\|_B\,\int_0^t\hat c(s)\,B(s)\,ds
\end{eqnarray*}
with a positive constant $C_0$ depending on the data. Applying the same technique to (\ref{rr-calz}), we get the following estimate
\begin{eqnarray*}
|\tilde{\bar z}(t)-\tilde z(t)|\leq C_1\,\|(\bar u-u,\bar z-z)\|_B\,\int_0^t\hat c(s)\,B(s)\,ds.
\end{eqnarray*}
{\it Step 5.} \emph{Estimate of $\tilde{\bar u}(t,x)-\tilde u(t,x).$} Denote $\alpha=\alpha[z](t,x)$ and $\bar\alpha=\alpha[\bar z](t, x).$
If $\alpha=\bar\alpha=0,$ then Assumptions $[\varphi],$ $[\lambda,3]$ and Step $3$ imply
\begin{eqnarray*}
|\tilde{\bar u}(t, x)-\tilde u(t,x)| \leq C_2\,\|(\bar u-u,\bar z-z)\|_B \,\int_0^t \hat c(s)\,B(s)\,ds.
\end{eqnarray*}
If $0<\alpha<\bar\alpha,$ then Assumption $[\lambda,3],$ previous Steps $2,$ $3,$ $4$ and Lemma \ref{war-4-5} imply
\begin{eqnarray*}
|\tilde{\bar u}(t,x)-\tilde u(t,x)|\leq C_3 \,\|(\bar u-u,\bar z-z)\|_B \,\int_0^t \hat c(s)\,B(s)\,ds.
\end{eqnarray*}
In the third case $0=\alpha<\bar\alpha$ (or $0=\bar\alpha<\alpha$) we consider the family of functions 
$$(u_\theta,z_\theta):=\theta(u,z)+(1-\theta)(\bar u,\bar z) \quad \mbox{for}\quad \theta\in[0,1].$$ 
We analyze the mapping 
$$
[0,1]\ni\theta\longmapsto\left(\alpha[z_\theta](t,x),\eta[z_\theta](\alpha[z_\theta](t,x);t,x)\right),
$$
whose values belong to the axes $0t$ and $0x.$ By the continuous dependence %on the parameter 
there exists $\theta\in[0,1]$ for which the point $(0,0)$ is attained. Then we have $\alpha[z_\theta](t,x)=0,$ thus $\eta[z_\theta](0;t,x)=0.$
Hence we get an intermediate point $({\tilde u}_{\theta},{\tilde z}_{\theta})=\mathcal{T}(u_{\theta},z_{\theta}),$ for which we have
$$
|\tilde u_\theta(t,x)-\tilde u(t,x)| \leq C_2\,\|(u_\theta-u,z_\theta-z)\|_B \,\int_0^t \hat c(s)\,B(s)\,ds
$$
and
$$
|\tilde{\bar u}(t,x)-\tilde{u}_\theta(t,x)| \leq C_3\,\|(\bar u-u_\theta,\bar z-z_\theta)\|_B \,\int_0^t \hat c(s)\,B(s)\,ds.
$$
Due to this observation we can reduce the estimate of $|\tilde{\bar u}(t, x)-\tilde u(t,x)|$ to the previous two cases 
\begin{eqnarray*}
|\tilde{\bar u}(t, x)-\tilde u(t,x)| &\leq& |\tilde{\bar u}(t, x)-\tilde u_\theta(t,x)|+|\tilde u_\theta(t, x)-\tilde u(t,x)|\\
&\leq& (C_2+C_3)\, \|(\bar u-u,\bar z-z)\|_B \,\int_0^t \hat c(s)\,B(s)\,ds.
\end{eqnarray*}
{\it Step 6.} \emph{The Bielecki norm.} Recall that $(\bar u,\bar z), (u,z)\in\mathcal{X}.$ 
In force of Steps $4,$ $5$ we derive
\begin{eqnarray*}
\frac{|\tilde{\bar z}(t)-\tilde z(t)|}{B(t)}
\leq
\|(\bar u-u,\bar z-z)\|_B\,\frac{C_1}{B(t)}\,\int_0^t \hat c(s)\,B(s)\,ds
\end{eqnarray*}
and
\begin{eqnarray*}
\frac{|\tilde{\bar u}(t,x)-\tilde u(t,x)|}{B(t)}
\leq
\|(\bar u-u,\bar z-z)\|_B\,\frac{C_2+C_3}{B(t)}\,\int_0^t \hat c(s)\,B(s)\,ds.
\end{eqnarray*}
From these relations we get
\begin{multline*}
\max\left\{\frac{|\tilde{\bar u}(t,x)-\tilde u(t,x)|}{B(t)},\frac{|\tilde{\bar z}(t)-\tilde z(t)|}{B(t)}\right\}\\
\leq \|(\bar u-u,\bar z-z)\|_B\,\frac{C_1+C_2+C_3}{B(t)}\,\int_0^t \hat c(s)\,B(s)\,ds.
\end{multline*}
Since we intend to estimate the right-hand side by $\Theta \|(\bar u -u, \bar z-z)\|_B,$ it suffices to solve the following elementary comparison equation
\begin{eqnarray*}
\Theta+(C_1+C_2+C_3)\,\int_0^t \hat c(s)\,B(s)\,ds=\Theta B(t).
\end{eqnarray*}
Its solution is given by
\begin{eqnarray*}
B(t)=\exp\left(\frac{C_1+C_2+C_3}{\Theta}\,\int_0^t \hat c(s)\,ds\right).
\end{eqnarray*}
Now it is seen that 
$$ \|(\tilde{\bar u}-\tilde u,\tilde{\bar z}-\tilde z)\|_B\leq \Theta\,\|(\bar u - u,\bar z-z)\|_B.$$
This completes the proof.
\end{proof}

\begin{remark}\label{rem_warbrzeg}
We explain how formula (\ref{rr-cal0}) can be regarded as a fixed point equation for the renewal condition (\ref{odn}). Based on (\ref{odn}) and (\ref{rr-cal}), we get
\begin{equation}\label{eq:tildeu}
\begin{split}\tilde u(t,0)&=\int_0^{\eta_0(t)} K(t,x)\,\tilde u_\alpha(\cdot,0)\,\exp\left(\int_{\alpha}^{t}\lambda\left(P_{t,x}^{u,z}(s)\right)\,ds\right)\,dx\\
&\quad +\,
\int^{\infty}_{\eta_0(t)}K(t,x)\,\varphi(\cdot,\eta[z](0;t,x))\,\exp\left(\int_0^{t}\lambda\left(P_{t,x}^{u,z}(s)\right)\,ds\right)\,dx
\end{split}
\end{equation}
for $(t,x)\in E,$ where $\alpha=\alpha[z](t,x).$ Using the appropriate change of variables, i.e. $\xi=\alpha[z](t,x)$ to the first integral in (\ref{eq:tildeu}) and $y=\eta[z](0;t,x)$ to the second integral, we obtain (\ref{rr-cal0}). Similar arguments apply to the function $\tilde z,$ for which we get the explicit formula
\begin{equation}\label{rr-calz}
\begin{split}
\tilde z(t)&=\int_0^t \tilde u(\xi,0)\,c(\xi,0,z_\xi)\,\exp\left(\int_\xi^t W\left(P_{\xi,0}^{u,z}(s)\right)\,ds\right)\,d\xi\\
&\quad +\, \int^{\infty}_0 \varphi(0,y)\,\exp{\left(\int_0^{t}W\left(P_{0,y}^{u,z}(s)\right)\,ds\right)}\,dy
\end{split}
\end{equation}
for $(t,x)\in E.$
These representations of $\tilde u(t,0)$ and $\tilde z(t)$ are useful in a priori estimates.
\end{remark}
%%%%%%%%%%%%%%%%%%%%%%%%%%%%%%%%%%%
\section*{Acknowledgments}
We are greatly indebted to the anonymous referees for their valuable comments and suggestions. They improved the whole text and made it more reader friendly.


\begin{thebibliography}{1}
\bibitem{B}
A. Bielecki: \textit{Une remarque sur la m\'ethode de Banach-Caciopoli-Tikhonov dans la th\'eorie des \'equations diff\'erentielles ordinaires,} Bulletin of the
Polish Academy of Sciences 4 (1956), 261-268.

\bibitem{KWB}
K. W. Blayneh: \textit{Analysis of age-structured host-parasitoid model,}
Far East J. Dynamical Systems 4, 2 (2002), 125-145.  

\bibitem{BC}
F. Brauer, C. Castillo-Ch\'avez: \textit{Mathematical Models in Population Biology and Epidemiology,} Springer-Verlag, New York,
2001.

\bibitem{CS}
A. Calsina, J. Salda\~na: \textit{A model of physiologically structured population dynamics with a nonlinear growth rate,} Jour. Mat. Biol. 33, (1995) 335-364.

\bibitem{Cus}
J. M. Cushing: \textit{Integrodifferential Equations and Delay Models in Population Dynamics,} Lect. Notes in Biomath. Vol. 20, Springer, Berlin, 1977.

\bibitem{Cus1}
J. M. Cushing: \textit{The dynamics of hierarchical age-structured populations,} J. Math. Biol. 32 (1994), 705-729.

\bibitem{DLo}
A. L. Dawidowicz, K. \L{}oskot: \textit{Existence and uniqueness of~solution of some integro-differential equation,} Ann. Polon. Math.
47 (1986), 79-87.

\bibitem{DPos}
A. L. Dawidowicz, A. Poskrobko: \textit{Age-dependent single-species population dynamics with delayed argument,} Math. Methods Appl. Sci. 33 (2010), 1122-1135.

\bibitem{EP}
B. Ebenman, L. Persson: \textit{Size-structured populations: ecology and evolution,} Springer-Verlag, 1988.

\bibitem{Far}
 J. Z. Farkas, D. M Green, P. Hinow: \textit{Semigroup analysis of structured parasite populations,} Mathematical Modelling of Natural Phenomena, 5 (2010), 94-114.
  

\bibitem{opt-cont}
G. Feichtinger, A. Prskawetz, V. M. Veliov: \textit{Age-structured optimal control in~population economics,} Theor. Pop. Biol. 65 (2004), 373-387.

\bibitem {VF}
H. von Foerster: \textit{Some remarks on changing populations,} in: The Kinetics of~Cellular Proliferation, Grune and Stratton, New
York, 1959, pp. 382-407. 

\bibitem{GK}
S. A. Gourley, Y. Kuang: \textit{A stage structured predator-prey model and its dependence on maturation delay and death rate,} J. Math. Biol. 49 (2004), 188-200.

\bibitem{GMC1}
M. E. Gurtin, R. C. McCamy: \textit{On the diffusion of biological populations,} Math. Bios. 33 (1977), 35-49.

\bibitem {GMC}
M. E. Gurtin, R. C. McCamy: \textit{Non-linear age-dependent Population Dynamics,} Arch. Rat. Mech. Anal. 54 (1974), 281-300.

\bibitem {Hale}
J. Hale, S. M. Verduyn Lunel: \textit{Introduction to Functional Differential Equations,} Springer Verlag, 1999.

\bibitem{Kato}
N. Kato: \textit{A general model of size-dependent population dynamics with nonlinear growth rate,} J. Math. Anal. Appl. 297 (2004) 234–256.

\bibitem{KM}
V. B. Kolmanovskii, A. Myshkis: \textit{Introduction to the theory and applications of functional differential equations,} MIA vol. 463, Kluwer Academic Publishers, Dordrecht, 1999.

\bibitem{Kraev}
E. A. Kraev: \textit{Existence and uniqueness for height structured hierarchical population models,} Natur. Resource Modeling, 14 (2001) 45–70.

\bibitem{Kuang}
Y. Kuang: \textit{Delay Differential Equations With Applications in Population Dynamics,} Academic Press, Boston, 1993.

\bibitem{LMW} 
A. Lasota, M. C. Mackey, M. Wa{\.z}ewska-Czy{\.z}ewska: \textit{Minimizing Therapeutically Induced Anemia,} J. Math. Biol. 13 (1981), 149-158.

\bibitem{lesl}
P. H. Leslie: \textit{The use of matrices in certain population mathematics,} Bio\-met\-rika, 33 (1945), 183-212.

\bibitem {lzw}
H. Leszczy\'nski, P. Zwierkowski: \textit{Existence of solutions to generalized von Foerster equations with functional dependence,}
Ann. Polon. Math. 83, 3 (2004), 201-210.

\bibitem {lesz}
H. Leszczy\'nski: \textit{Differential functional von Foerster equations with renewal,} Cond. Mat. Phys. 54, 11 (2008), 361-370.

\bibitem{MRud}
M. C. Mackey, R. Rudnicki: \textit{Global stability in a delayed partial differential equation describing cellular replication,} J. Math. Biol. 33 (1994), 89-109.

\bibitem{MD}
J. A. J. Metz, O. Diekmann: \textit{The Dynamics of Physiologically Structured Populations,} Lecture Notes in Biomathematics. Vol. 68.
Springer-Verlag, Berlin, Heidelberg, 511 pp., 1986.

\bibitem{Per}
B. Perthame: \textit{Transport Equations in Biology,} Frontiers in Mathematics, Birkhauser, 2007.

\bibitem{PSW}
J. A. Powell, I. Slapni\v car, W. van der Werf: \textit{Epidemic spread of a lesion-forming plant pathogen analysis of a mechanistic model with infinite age structure,} J. Linear Algebra Appl. 398 (2005), 117-140.
 
\bibitem{Tch}
J. M. Tchuenche: \textit{An age-structured model with delay mortality,} BioSystems 81 (2005), 255-260.

\bibitem{Walter}
W. Walter: \textit{Ordinary Differential Equations,} Springer-Verlag, New York, Berlin, Heidelberg, 1998.

\bibitem{Webb}
G. F. Webb: \textit{Theory of nonlinear age-dependent population dynamics,} Marcel Dekker, New York, 1985.

\bibitem{Wu}
J. Wu: \textit{Theory and Applications of Partial Functional Equations,} Springer-Verlag, New York, Berlin, Heidelberg, Tokyo, 1996.  

\end{thebibliography}
\end{document}